\documentclass[graybox]{svmult}

\usepackage{todonotes}

\usepackage{mathptmx}       
\usepackage{helvet}         
\usepackage{courier}        
\usepackage{type1cm}        
\usepackage{makeidx}         
\usepackage{graphicx}        
\usepackage{multicol}        
\usepackage[bottom]{footmisc}

\usepackage{amssymb,enumerate}

\newcommand{\ind}{\mathbf{1}}
\newcommand{\vol}{{\rm  vol}}
 \newcommand{\Borel}{\mathcal{B}}

\newcommand{\N}{\mathbb{N}}
\newcommand{\R}{\mathbb{R}}
\newcommand{\expect}{{\rm E}}
\newcommand{\abs}[1]{\left\vert #1 \right\vert}	
\newcommand{\norm}[1]{\left\Vert #1 \right\Vert}	

\renewcommand{\a}{\alpha}

\newcommand{\eps}{\varepsilon}

\renewcommand{\Pr}{\mbox{Pr}}
\newcommand{\dint}{\mbox{\rm d}}

\newcommand{\ball}{B_d}
\newcommand{\Gap}{ \mbox{\rm gap}}
\newcommand{\supp}{ \mbox{\rm supp}}
\newcommand{\diag}{ \mbox{\rm diag}}
\newcommand{\tr}{ \mbox{\rm tr}}

\makeindex             

\begin{document}

\title*{Hit-and-run for numerical integration}
\author{Daniel Rudolf}
\institute{Dr. Daniel Rudolf \at Friedrich-Schiller-University Jena, 
Institute of Mathematics, Ernst-Abbe-Platz 2, 07743 Jena, Germany. \email{daniel.rudolf@uni-jena.de}}
%
%

\maketitle

\abstract{
We study the numerical computation of an expectation of a bounded 
function $f$ with respect to a measure given 
by a non-normalized density on a convex body $K\subset\R^d$. 
We assume that the density is log-concave, satisfies a variability condition and is not too narrow.
In \cite{MaNo07,Ru09,Ru12} it is required that $K$ is the Euclidean unit ball. 
We consider general convex bodies or even the whole $\R^d$ 
and show that the integration problem satisfies a refined form of tractability.
The main tools are the hit-and-run algorithm and an error bound 
of a multi run Markov chain Monte Carlo method.
}

%

\section{Introduction and results} 
\label{sec: intro}

In many applications, for example in Bayesian inference, see \cite{BrGeJoMe11,GiRiSp96}, or in statistical 
physics, see \cite{Mar04,So97}, it is desirable to compute an expectation of the form
\[
  \int_K  f(x)\, \pi_\rho (\dint x) = \int_K f(x)\, c\,\rho(x)\, \dint x, 
\]
where the probability measure $\pi_\rho$ is given by the density $c\,\rho$ with $c>0$. 
The normalizing constant of the density
\[
  \frac{1}{c} = \int_K \rho(x)\, \dint x
\]
is not known and hard to compute. We want to have algorithms 
that are able to compute the expectation without any
precompution of $c$.

More precisely, let $\rho\colon \R^d \to \R_+$ be a 
possibly non-normalized density function,
let $K=\supp(\rho)\subset \R^d$ be a convex body 
and let $f\colon K \to \R$ be integrable with respect to $\pi_\rho$.
For a tuple $(f,\rho)$ we define the desired quantity
\begin{equation}  \label{eq: sol}
  A(f,\rho) = \frac{\int_K f(x)\,\rho(x)\, \dint x}{\int_K \rho(x)\,\dint x}. 
\end{equation}
In \cite{MaNo07} a simple Monte Carlo method is considered which evaluates the numerator and denominater of 
$A(f,\rho)$ on a common independent, uniformly distributed sample in $K$. 
There it must be assumed that one can sample the uniform distribution in $K$.
The authors show
that this algorithm is not able to use any additional structure, such as log-concavity, of the density function.
But they show that such structure 
can be used by Markov chain Monte Carlo which then outperforms the simple Monte Carlo method.\\

Markov chain Monte Carlo algorithms for the
integration problem of the form $(\ref{eq: sol})$ are considered in \cite{MaNo07,NoWo10,Ru09,Ru12}.
Basically it is always assumed that $K$ is the Euclidean unit ball rather than a general convex body.
We extend the results to the case where $K$ might 
even be the whole $\R^d$ if the density satisfies some further properties. 
We do not assume that we can sample with respect to $\pi_\rho$.   
The idea is to compute $A(f,\rho)$ by using a Markov chain which approximates $\pi_\rho$. 
We prove that the integration problem (\ref{eq: sol}) satisfies an extended type of 
tractability.
Now let us introduce the error criterion and the new notion of tractability. \\

\emph{Error criterion and algorithms.}
Let $t\colon \N \times \N \to \N$ be a function and let $A_{n,n_0}$ be a generic algorithm which uses 
$t(n,n_0)$ Markov chain steps. Intuitively, the number $n_0$ determines the number of steps 
to 
approximate $\pi_\rho$.
The number $n$ determines the number of pieces of information of $f$ used by the algorithm. 
The error is measured in mean square sense, for a tuple $(f,\rho)$ it is given by
\[
e(A_{n,n_0}(f,\rho)) = \left(\expect \abs{A_{n,n_0}(f,\rho)-A(f,\rho)}^2 \right)^{1/2},
\]
where $\expect$ denotes the expectation with respect to the joint distribution of the used sequence of random variables
determined by the Markov chain. 

For example the algorithm might be a single or multi run Markov chain Monte Carlo.
More precisely, assume that we have a Markov chain with limit distribution $\pi_\rho$ and let $X_1,\dots,X_{n+n_0}$ be 
the first $n+n_0$ steps. Then
\[
  S_{n,n_0}(f,\rho) = \frac{1}{n} \sum_{j=1}^n f(X_{j+n_0})
\]
is an approximation of $A(f,\rho)$ and the function $t(n,n_0)=n+n_0$. 
In contrast to the single run Markov chain Monte Carlo $S_{n,n_0}$ one might consider a multi run Markov chain Monte Carlo, 
say $M_{n,n_0}$, given as follows.
Assume that we have $n$ independent Markov chains with the same transition kernel, the same initial distribution 
and limit distribution $\pi_\rho$. Let $X^1_{n_0},\dots,X^n_{n_0}$
be the sequence of the $n_0$th steps of the Markov chains,
then
\[
   M_{n,n_0}(f,\rho)   = \frac{1}{n}  \sum_{j=1}^n f(X_{n_0}^j)
\]
is an approximation of $A(f,\rho)$. In this setting the function $t(n,n_0)=n\cdot n_0$.\\

\emph{Tractability.}
In \cite{MaNo07,NoWo10} a notion of tractability for the integration problem (\ref{eq: sol}) is introduced. 
It is assumed that $\norm{f}_{\infty}\leq 1$ and that the density function satisfies 
\[
  \frac{\sup_{x\in K} \rho(x)}{\inf_{x\in K} \rho(x)} \leq \gamma, 
\]
for some $\gamma\geq 3$. Let $s_{\eps,\gamma}(n,n_0)$ be the minimal number of function values 
of $(f,\rho)$ to guarantee an $\eps$-approximation with respect to the error above. 
Then the integration problem is called tractable with respect to $\gamma$ if $s_{\eps,\gamma}(n,n_0)$
depends polylogarithmically on $\gamma$ and depends polynomially on $\eps^{-1}$, $d$.   
We extend this notion of tractability.
We study a class of tuples $(f,\rho)$ which satisfy
 $\norm{f}_{\infty}\leq 1$ and we assume that for any $\rho$ there exists a set $G\subset K$ 
such that for $\kappa\geq 3$ holds  
\begin{equation}  \label{eq: kappa_G}
  \frac{\int_K \rho(x)\, \dint x}{\vol_d(G)\;\inf_{x\in G} \rho(x)} \leq \kappa,
\end{equation}
where $\vol_d(G)$ denotes the $d$-dimensional volume of $G$.
Then we call the integration problem tractable with respect to $\kappa$ if the minimal
number of function values $t_{\eps,{\kappa}}(n,n_0)$ of $(f,\rho)$ to guarantee an $\eps$-approximation 
satisfies for some non-negative numbers
$p_1$, $p_2$ and $p_3$ that
\[
  t_{\eps,{\kappa}}(n,n_0) = \mathcal{O}( \eps^{-p_1} d^{p_2} [\log {\kappa}]^{p_3} ),
  \quad \eps>0,\; d\in\N,\; {\kappa}\geq3.
\]
Hence we permit only polylogarithmical dependence on the number ${\kappa}$, 
since it might be very large 
(e.g. $10^{30}$ or $10^{40}$).
The extended notion of tractability allows us to consider $K=\supp(\rho) = \R^d$. 
\\

The structure of the work and the main results are as follows. 
We use the hit-and-run algorithm to approximate $\pi_\rho$. 
An explicit estimate of the total variation distance of the hit-and-run algorithm, 
proven by Lov{\'a}sz and Vempala in \cite{LoVe06-1,LoVe06}, 
and an error bound of the mean square error of $M_{n,n_0}$ are essential.
In Section~\ref{sec: MC_err_bound} we provide the basics on Markov chains and prove an error bound of $M_{n,n_0}$.
In Section~\ref{sec: densities} we define the class of density functions.
Roughly we assume that the densities are log-concave, that for any $\rho$ there exists a set $G\subset K$ such that
condition $(\ref{eq: kappa_G})$ holds for $\kappa\geq3$ and that the densities are not too narrow.
Namely, we assume that level sets of $\rho$ of measure larger than $1/8$ contain a ball with radius $r$.
We distinguish two settings which guarantee that the densities are not too spread out. 
Either the convex body $K=\supp (\rho)$ is bounded by a ball with radius $R$ around $0$, 
then we say $\rho\in \mathcal{U}_{r,R,\kappa}$, 
or the support of $\rho$ is bounded in average sense, 
\[
  \int_K \abs{x-x_\rho}^2 \pi_\rho(\dint x) \leq 4 R^2, 
\]
where  $x_\rho = \int_K x \, \pi_\rho(\dint x) \in \R^d$ is the centroid. Then we say $\rho\in \mathcal{V}_{r,R,\kappa}$.
For precise definitions see Section~\ref{sec: densities}. 
In Section~\ref{sec: har} we provide the hit-and-run algorithm and state convergence properties 
of the algorithm for densities from $\mathcal{U}_{r,R,\kappa}$ and $\mathcal{V}_{r,R,\kappa}$.   
Then we show that the integration problem (\ref{eq: sol}) is tractable with respect to $\kappa$, see Section~\ref{sec: main_res}.
For $\rho\in \mathcal{U}_{r,R,\kappa}$ we obtain in Theorem~\ref{thm: Err} that
  \begin{equation}  \label{eq: res_1}
    t_{\eps,{\kappa}}(n,n_0)
 =\mathcal{O} (d^{2}\,[\log d]^{2} \,\eps^{-2}\,[\log \eps^{-1}]^3
		    \, [\log\kappa]^3).
  \end{equation}
For $\rho\in \mathcal{V}_{r,R,\kappa}$ we find in Theorem~\ref{thm: gen_Err} a slightly worse bound of the form 
  \begin{equation}  \label{eq: res_2}
    t_{\eps,{\kappa}}(n,n_0)
 =\mathcal{O} (d^{2}\,[\log d]^2 \,\eps^{-2}\,[\log \eps^{-1}]^5
		    \, [\log\kappa]^5).
  \end{equation}
Here the $\mathcal{O}$ notation hides the polynomial dependence on $r$ and $R$. 

In \cite{MaNo07,NoWo10,Ru09,Ru12} it is proven that the problem (\ref{eq: sol}) is tractable with respect to $\gamma$ 
for $K=\ball$, where $\ball$ denotes the Euclidean unit ball. Note that for $G=\ball$ 
we have 
\[
    \frac{\int_K \rho(x)\, \dint x}{\vol_d(G)\;\inf_{x\in G} \rho(x)} 
 \leq \frac{\sup_{x\in K} \rho(x)}{\inf_{x\in K} \rho(x)} \leq \gamma.
\]
Furthermore it is assumed that $\rho\colon \ball \to \R_+$ is log-concave and $\log \rho$ is Lipschitz. 
Then the Metropolis algorithm with a ball walk proposal is used to approximate $\pi_\rho$. 
For $\norm{f}_{p}\leq 1$ with $p>2$ the algorithm $S_{n,n_0}$ is considered for the approximation of $A(f,\rho)$.
It is proven that
\begin{equation} \label{eq: compl_gamma}
    s_{\eps, \gamma}(n,n_0) = \mathcal{O}(d \max\{\log^2(\gamma),d\}(\eps^{-2}+\log \gamma)).
\end{equation}
In open problem $84$ of \cite{NoWo10} it is asked whether one can extend 
this result
to other families of convex sets. 
The complexity bound of (\ref{eq: compl_gamma}) is better than the results of (\ref{eq: res_1}) and (\ref{eq: res_2}) 
in terms of the dimension, the precision and $\gamma$. 
On the one hand the assumption that $K=\ball$ is very restrictive but on the other hand the estimates of 
(\ref{eq: res_1}) and (\ref{eq: res_2}) seem to be pessimistic. 
However, with our results we contribute to problem $84$ 
in the sense that tractability with respect to $\gamma$ can be shown 
for arbitrary convex bodies or even the whole $\R^d$ if the density functions satisfy certain properties.    

\section{Markov chains and an error bound}  \label{sec: MC_err_bound}
Let $(X_n)_{n\in\N}$ be a Markov chain with transition kernel $P(\cdot,\cdot)$ and initial distribution $\nu$
on a measurable space $(K,\Borel(K))$, where $K\subset \R^d$ and $\Borel(K)$ is
the Borel $\sigma$-algebra. We assume that the transition kernel $P(\cdot,\cdot)$ is reversible with respect to $\pi_\rho$. 
For $p\in[1,\infty]$ we denote by $L_p=L_p(\pi_\rho)$ the class of functions $f\colon K \to \R$ with 
\[
\norm{f}_{p}=\left(\int_K \abs{f(x)}^p\,\pi_\rho(\dint x)\right)^{1/p} < \infty.
\]
Similarly we denote by $\mathcal{M}_p$ 
the class of measures $\nu$ which are absolutely continuous with respect to $\pi_\rho$ and
where the density $\frac{d\nu}{d\pi_\rho}\in L_p$.
The transition kernel induces an operator $P \colon L_p \to L_p$ given by
\[
    Pf(x)=\int_K f(y)\,P(x,\dint y), \quad x\in K,
\] 
and it induces an operator $P \colon \mathcal{M}_p \to \mathcal{M}_p$ given by
\[
  \mu P (C) = \int_K P(x,C) \, \mu(\dint x),\quad C\in\Borel(K).
\]
For $n\in\N$ and a probability measure $\nu$ note that
$  \Pr(X_n\in C) = \nu P^n(C)$, where $C\in\Borel(K)$. 
We define the total variation distance
between $\nu P^n$ and $\pi_\rho$ as
\[
    \norm{\nu P^n - \pi_\rho}_{\mbox{tv}} = \sup_{C\in\Borel(K)} \abs{\nu P^n(C)-\pi_\rho(C)}.
\]
Under suitable assumptions on the Markov chain one obtains that 
$\norm{\nu P^n - \pi_\rho}_{\mbox{tv}} \to 0$ as $n\to \infty$.


Now we consider the multi run Markov chain Monte Carlo method and prove an error bound.
This bound is not new, see for example \cite{BeCh09}.

\begin{theorem}  \label{thm: err_bound_m}
Assume that we have $n_0$ independent Markov chains with 
transition kernel $P(\cdot,\cdot)$ and initial distribution $\nu\in\mathcal{M}_1$. 
Let $\pi_\rho$ be a stationary distribution of $P(\cdot,\cdot)$.
Let $X^1_{n_0},\dots,X^n_{n_0}$
be the sequence of the $n_0$th steps of the Markov chains
and let
\[
   M_{n,n_0}(f,\rho)   = \frac{1}{n}  \sum_{j=1}^n f(X_{n_0}^j).
\]
Then
\[
  e(M_{n,n_0}(f,\rho))^2 \leq \frac{1}{n} \norm{f}_{\infty}^2 + 2 \norm{f}_{\infty}^2 
\norm{\nu P^n - \pi_\rho}_{\mbox{tv}}.
\] 
\end{theorem}
\begin{proof}
 With an abuse of notation let us denote 
\[
 A(f)=\int_K f(x)\, \pi_\rho(\dint x)
\quad \mbox{and} \quad \nu P^{n_0} (f) = \int_K f(x)\, \nu P^{n_0}(\dint x).
\]
 We decompose the error into variance and bias. 
 Then
 \begin{eqnarray*}
   e(M_{n,n_0}(f,\rho))^2 & = &
  \frac{1}{n} \int_K \abs{f(x)-\nu P^{n_0} (f)}^2 \nu P^{n_0}(\dint x) + \abs{\nu P^{n_0}(f)-A(f)}^2 \\
 & = & \frac{1}{n}  \left( \nu P^{n_0}(f^2)-\nu P^{n_0}(f)^2 \right)
	+ \abs{\nu P^{n_0}(f)-A(f)}^2 \\
 &  \leq &  \frac{1}{n}  \norm{f}_{\infty}^2 + \int_K f(x)^2 \abs{\nu P^{n_0}(\dint x)- \pi_\rho(\dint x)}\\
 & \leq & \frac{1}{n}  \norm{f}_{\infty}^2 + 2 \norm{f}_{\infty}^2 \norm{\nu P^{n_0}-\pi_\rho}_{\mbox{tv}}.
  \end{eqnarray*}
 The last inequality follows by a well known characterization of the total variation distance, see 
 for example \cite[Proposition~3]{RoRo04}.
\end{proof}

Very often there exists
a number $\beta\in[0,1)$ and a number $C_\nu<\infty$ such that
\[
  \norm{\nu P^n - \pi_\rho}_{\mbox{tv}} \leq C_\nu \beta^n.
\]
For example, if $\beta= \norm{P- A}_{L_2\to L_2}<1$ and 
$C_\nu = \frac{1}{2} \norm{\nu-\pi_\rho}_{2}$, see \cite{RoRo97} for more details.
Let us define the $L_2$-spectral gap as
\[
\Gap(P)=1-\norm{P-A}_{L_2\to L_2}.
\]
This is a significant quantity, see for instance \cite{Al87,Ru12,So97,Ul11,Ul12}. 
In \cite{Ru12} it is shown that 
\[
    e(S_{n,n_0}(f,\rho))^2 \leq \frac{4 \norm{f}_{4}}{n\, \Gap(P)} 
  \qquad \mbox{for} \qquad 
  n_0 \geq \frac{\log \left(64 \norm{\frac{d\nu}{d\pi_\rho}-1}_{2}\right)}{ \Gap(P)} .
\]
There are several Markov chains where it is possible to provide, 
for certain classes of density functions,
 a lower bound of $\Gap(P)$
which grows polynomially with respect to the dimension, see for example \cite{LoVe06,MaNo07}. 
Then, the error bound of the single run Markov chain Monte Carlo method might
imply that the integration problem $(\ref{eq: sol})$ is 
tractable with respect to some $\kappa$.

Note that there are also other possible approximation schemes and other bounds of the error of $S_{n,n_0}$   
which depend on different assumptions to the Markov chain
(e.g. Ricci curvature condition, drift condition, small set), 
see for instance \cite{JoOl10,LatMiNi11-1,LatMiNi11,LatNi11}. 
For example one might consider a multi run Markov chain Monte Carlo method 
where function values of a trajectory of each Markov chain after a sufficiently large $n_0$ are used. 
But all known error bounds of such methods include quantities
such as the $L_2$-spectral gap or the conductance.

It is not an easy task to prove that a Markov chain
satisfies the different assumptions stated above
and it is also not an easy task to prove 
a lower bound of the $L_2$-spectral gap.
It might be easier to estimate 
the total variation distance of $\nu P^{n_0}$ and $\pi_\rho$ directly.
Then one can
use Theorem~\ref{thm: err_bound_m}
to show that the integration problem $(\ref{eq: sol})$ is tractable with respect to
some $\kappa$.

\section{Densities with additional structure}
\label{sec: densities}

Let us assume that the densities have some additional structure.  
For $0<r\leq R$ and $\kappa\geq 3$ a density function $\rho \colon K \to  \R_+$ is in $\mathcal{U}_{r,R,\kappa}$ 
if the following properties are satisfied:
\begin{enumerate}[(a)]
\item \label{en: log_con}
  $\rho$ is log-concave, i.e. for all $x,y\in K$ and $\lambda\in[0,1]$ one has
  \[
    \rho(\lambda x+(1-\lambda)y) \geq \rho(x)^\lambda \rho(y)^{1-\lambda}.
  \]
\item \label{en: bound} 
  $\rho$ is strictly positive, i.e. 
  $K=\supp (\rho) $ 
 and we assume that $K\subset R \ball$, where $R\ball$ is the Euclidean ball with radius $R$ around $0$.
\item \label{en: init}
  There exists a set $G\subset K$ such that

\[
 \frac{\int_K \rho(x)\, \dint x}{\vol_d(G)\;\inf_{x\in G} \rho(x)} \leq \kappa,
\]
and we can sample the uniform distribution on $G$.
 \item \label{en: lev_set} 
       For $s>0$ let $K(s)=\{ x\in K \mid \rho(x)\geq t \}$ be the level set of $\rho$
       and let $B(z,r)$ be the Euclidean ball with radius $r$ around $z$.
       Then
  \[
      \pi_\rho(K(s)) \geq \frac{1}{8}  \quad \Longrightarrow \quad \exists z \in K\quad B(z,r)\subset K(s). 
  \]
\end{enumerate}
  
The log-concavity of $\rho$ implies that the maximal value is attained on a convex set, 
that the function is continuous and that one has an isoperimetric inequality, see \cite{LoVe06}. 
Assumption (\ref{en: bound}) gives that $K$ is bounded. 

By (\ref{en: init}) we can sample the uniform distribution on $G$.
We can choose it as 
initial distribution for a Markov chain, where the number 
$\kappa$ provides an estimate of the influence of this 
initial distribution. 

The condition on the level set $K(s)$ guarantees that the peak is not too narrow.  
Roughly speaking, if the $\pi_\rho$ measure of a level set is not too small, then 
the Lebesgue measure is also not too small. 
Note that $K$ is bounded from below, since condition (\ref{en: lev_set}) implies that $B(z,r)\subset K$. 


Now we enlarge the class of densities. Let us define the following property:
\begin{enumerate}[(a')]
 \addtocounter{enumi}{1}
 \item \label{en: expect_bound}  
  $\rho$ is strictly positive, i.e. 
  $K=\supp (\rho) $ 
  and $x_\rho = \int_K x\; \pi_\rho(\dint x) \in \R^d$ is the centroid of $\pi_\rho$.
        Then
      \[
	 \int_K \abs{x-x_\rho}^2 \pi_\rho(\dint x) \leq 4\, R^2.
      \]
\end{enumerate}
We have $\rho \in \mathcal{V}_{r,R,\kappa}$ if the density $\rho$ satisfies 
(\ref{en: log_con}), (\ref{en: expect_bound}'), (\ref{en: init}) and (\ref{en: lev_set}). 
We substituted the boundedness condition (\ref{en: bound}) by (\ref{en: expect_bound}').
Note that (\ref{en: bound}) implies (\ref{en: expect_bound}'). 
Hence $\mathcal{U}_{r,R,\kappa} \subset \mathcal{V}_{r,R,\kappa}$. 
Condition (\ref{en: expect_bound}') provides a boundedness criterion in average sense.
Namely, it implies that
\[
  \int_K \int_K \abs{x-y}^2 \pi_\rho(\dint x)\; \pi_\rho(\dint y) \leq 8 R^2.
\]
\\

\emph{Example of a Gaussian function in $\mathcal{V}_{r,R,\kappa}$}.
Let $\Sigma$ be a symmetric and positive definite $d\times d$ matrix.
We consider the non-normalized density
\[
    \varphi(x)= \exp(-\frac{1}{2}\;x^T \Sigma^{-1} x ), \quad x\in \R^d.
\]
The target distribution
$\pi_\varphi$ is a normal distribution
with mean $x_\varphi=0\in\R^d$ and covariance matrix $\Sigma$.
There exists an orthogonal matrix $V=(v_1,\dots,v_d)$, where $v_1,\dots,v_d$ are the eigenvectors of $\Sigma$.
Then
\[
V^{-1} \Sigma V = \Lambda,  
\]
where $\Lambda=\diag(\lambda_1,\dots, \lambda_d)$ and $\lambda_1,\dots,\lambda_d$ 
with $\lambda_i>0$ for $i\in\{1,\dots,d\}$ are the corresponding eigenvalues of $\Sigma$.
Recall that the trace 
and the determinant of $\Sigma$ are
\[
 \tr(\Sigma)= \sum_{i=1}^d \lambda_i
  \quad \mbox{and} \quad
 \det(\Sigma) =  \prod_{i=1}^d\, \lambda_i.
\]
We show that if $r$, $R$ and $\kappa$ are appropriately chosen, then $\varphi\in\mathcal{V}_{r,R,{\kappa}}$. 
\begin{description} 
  \item[To (\ref{en: log_con}):] The density $\varphi$ is obviously log-concave.
  \item[To (\ref{en: expect_bound}'):]  Since $x_\varphi=0$ we obtain
\[
  \int_K \abs{x-x_\varphi}^2 \pi_\varphi(\dint x) 
=\frac{1}{(2\pi)^{d/2} \sqrt{\det(\Sigma)}} \int_{\R^d} \abs{x}^2 \varphi(x)\,\dint x = \tr(\Sigma).
\] 
Hence we set $R=\frac{1}{2} \sqrt{\tr(\Sigma)}$.
  \item[To (\ref{en: init}):] Let $\lambda_{\rm{min}}= \min_{i=1,\dots,d} \lambda_i$ and let $v_{\rm{min}}$ 
    be the corresponding eigenvector. 
     Note that $x^T \Sigma^{-1} x \leq \lambda_{\rm{min}}^{-1} \abs{x}^2$ and that equality holds for $x=v_{\rm{min}}$.
With $G=\ball$ we obtain
\[
\frac{\int_{\R^d} \varphi(x)\, \dint x}{\vol_d(\ball)\;\inf_{x\in \ball} \varphi(x)} 
= \exp(\frac{1}{2}\;\lambda_{\rm{min}}^{-1})\; \Gamma(d/2+1)\;2^{d/2}\sqrt{\;\det(\Sigma)} ,
\]
where $\Gamma(d)=\int_0^\infty t^{d-1} \exp(-t)\,\dint t$ is the gamma function. Hence we set
\[
  \kappa = \exp(\frac{1}{2}\;\lambda_{\rm{min}}^{-1})\; \Gamma(d/2+1)\;2^{d/2}\sqrt{\;\det(\Sigma)}.
\]
  \item[To (\ref{en: lev_set}):]  The level sets of $\varphi$ are ellipsoids 
\[
  K(s)= \{ x\in \R^d\mid x^T \Sigma^{-1} x \leq 2\log(s^{-1})	\}, \quad s\in[0,1].
\]
In general one has
\[
  \pi_\varphi(K(s)) = \frac{\int_0^\infty \vol_d(K(s)\cap K(t))\;\dint t}{\int_0^\infty \vol_d(K(t))\; \dint t} 
= \frac{s\;\vol_d(K(s))+ \int_s^\infty \vol_d(K(t))\, \dint t}{\int_0^\infty \vol_d(K(t))\; \dint t}.
\]
By the well known formula of the volume of an ellipsoid we obtain
\[
  \vol_d(K(t)) = 2^{d/2}\;\log^{d/2}(t^{-1}) \sqrt{\det(\Sigma)}\; \vol_d(\ball) ,\quad t\in[0,1]
\] 
and 
\[
  \pi_\varphi(K(s)) 
= \frac{s \;\log^{d/2}(s^{-1})+\int_s^1 \log^{d/2}(t^{-1})\;\dint t }{\int_0^1 \log^{d/2}(t^{-1})\;\dint t}, 
  \quad s\in[0,1].
\]
Hence
\[
  \pi_\varphi(K(s))= \frac{\gamma(\log s^{-1},d/2)}{\Gamma(d/2)}, \quad s\in[0,1],
\]
where $\gamma(r,d)=\int_0^r t^{d-1} \exp(-t) \,\dint t$ is the lower incomplete gamma function.
Let us define a function $r^*:\N\to\R$ by
\[
  r^*(d)=\inf	\{	r\in[0,\infty) \colon\; \gamma(r,d/2) \geq \frac{1}{8}\, \Gamma(d/2)	\}.
\]
If we substitute $1/8$ by $1/2$ in the definition of $r^*(d)$ 
we have the median of the gamma distribution with parameter $d/2$ and $1$.
It is known that the median is in $\Theta(d)$, see \cite{AdJo05}.
Figure~\ref{fig: r_star_d} suggests that $r^*(d)$ behaves also linearly in $d$.

\begin{figure} 
 \centering
 \includegraphics[width=10cm]{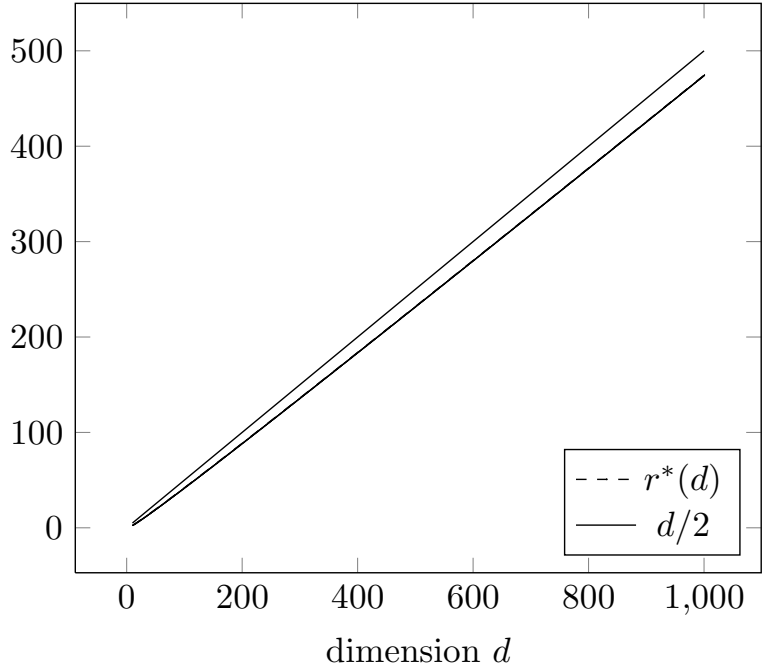}
\caption{Plot of an approximation of $r^*(d)$ with a Newton method
and an appropriately chosen initial value.}
\label{fig: r_star_d}
\end{figure}

Let $\log (s^*(d)^{-1}) = r^*(d)$, such that $s^*(d)=\exp(-r^*(d))$. Then
\[
  \pi_\varphi(K(s^*(d))) = \frac{1}{8}\quad \mbox{and} \quad B(0,(\lambda_{\rm{min}} r^*(d))^{1/2}) \subset K(s^*(d)).
\]
Hence we set $r=(\lambda_{\rm{min}} r^*(d))^{1/2}$.
\end{description}
Let us summarize. For $r=(\lambda_{\rm{min}} r^*(d))^{1/2}$, $R= \frac{1}{2} \sqrt{\tr(\Sigma)} $ 
and
\[
  \kappa = \exp(\frac{1}{2}\;\lambda_{\rm{min}}^{-1})\; \Gamma(d/2+1)\;2^{d/2}\sqrt{\;\det(\Sigma)}
\]
we obtain that $\varphi\in \mathcal{V}_{r,R,{\kappa}}$. Note that ${\kappa}$ depends exponentially
on the dimension $d$. However, if one has tractability with respect to ${\kappa}$, 
then the error depends polynomially on the dimension.


\section{Hit-and-run algorithm}  \label{sec: har}
For $\rho \colon K\to \R_+$ the hit-and-run algorithm is as follows.
Let $\nu$ be a probability measure on $(K,\Borel(K))$ and let $x_1\in K$ be chosen by $\nu$. 
For $k\in\N$ suppose that the states $x_1,\dots,x_k$ are already computed. Then 
\begin{enumerate}
 \item choose a direction $u$ uniformly distributed on $\partial \ball$;
 \item set $x_{k+1}=x_k+ \alpha \,u$, where $\alpha \in I_k=\{\a \in \R \mid x_k+\a u \in K\}$ is chosen with respect to the distribution determined by
       the density
	\[
	    \ell_k(s) = \frac{\rho(x_k+s\, u)}{\int_{I_k}  \rho(x_k+t\, u)\,\dint t}, \quad s \in I_k.
	\]
\end{enumerate}

The second step might cause implementation issues.
However, if we have a log-concave density $\rho$ then $\ell_k$ is also log-concave.
In this setting one can use different acceptance/rejection methods.
For more details see for example \cite[Section~2.4.2]{CaRo04} or \cite{LoVe07}.
In the following we assume that we can sample the distribution determined by $\ell_k$. 

Other algorithms for the approximation of $\pi_\rho$ would be a Metropolis algorithm with suitable proposal \cite{MaNo07} 
or a combination of a hit-and-run algorithm with uniform stationary distribution and a Ratio-of-uniforms method \cite{KaLePo05}.
Also hybrid samplers are promising methods, 
especially when $\rho$ decreases exponentially in the tails \cite{FoMoRoRo03}. 

Now let us state the transition kernel,  
say $H_\rho$, of the hit-and-run algorithm 
\[
  H_\rho(x,C) \;=\; \frac{2}{\vol_{d-1}(\partial B^ d)} 
  \int_C \frac{ \rho(y)\, \dint y}{\ell_\rho(x,y)\abs{x-y}^{d-1}},
  \quad x\in K,\,C\in \mathcal{B}(K),
\]
where 
\[
\ell_\rho(x,y)\;=\; \int_{-\infty}^{\infty} 
\rho(\lambda x + (1-\lambda)y) \ind_K(\lambda x + (1-\lambda)y)\, \dint \lambda.
\] 
The transition kernel $H_\rho$ is reversible with respect to $\pi_\rho$, let us refer to \cite{BeRoSm93}
for further details. 

In the following we state several results from Lova\'sz and Vempala. 
This part is based on \cite{LoVe06-1}. 
We start with a special case of \cite[Theorem~1.1]{LoVe06-1} and
sketch the proof of this theorem.

\begin{theorem} \label{thm: est_n_tv}
 Let $\eps\in(0,1/2)$ and $\rho\in \mathcal{U}_{r,R,\kappa}$. Let $\nu$ be an initial distribution with
 the following property. There exists a set $S_\eps\subset K$ 
 and a number $D\geq1$ such that
 \[
    \frac{d\nu}{d\pi_\rho}(x) \leq D, \quad x\in K\setminus S_\eps,
 \]
 where $\nu(S_\eps)\leq \eps$. Then for 
 \[
    n_0> 10^{27} (d r^{-1}\, R)^2 \log^2(8\,D\, d  r^{-1}\, R  \eps^{-1}) \log(4\,D\, \eps^{-1})
 \]
 the total variation distance between $\nu H_\rho^{n_0}$ and $\pi_\rho$ is less than $2\eps$.
\end{theorem}

\begin{proof}[Sketch]
\begin{enumerate}
 \item Let us assume that $S_\eps=\emptyset$:\\
  Then it follows
  $
      \norm{\frac{d\nu}{d\pi_\rho}}_{\infty} \leq D,
  $ 
so that $\nu\in \mathcal{M}_\infty$. 
 We use \cite[Corollary~1.6]{LoSi93} with $s=\frac{\eps}{2 D}$ and obtain
\[
        \norm{\nu H_\rho^n - \pi_\rho}_{\mbox{tv}} \leq \eps/2 + D \exp({-\frac{1}{2}\,n\; \Phi^2_{\frac{\eps}{2D}} }),
\]
 where $\Phi_{\frac{\eps}{2 D}}$ is the $\frac{\eps}{2D}$-conductance
 of $H_\rho$. By Theorem~3.7 of \cite{LoVe06-1} and the scaling invariance of the hit-and-run algorithm we find 
 a lower bound of $\Phi_{\frac{\eps}{2D}}$. 
 It is given by
  \begin{equation} \label{eq: low_scond}
       \Phi_{\frac{\eps}{2 D}} \geq \frac{10^{-13}}{2\;d r^{-1}\, R  \log(4\,d r^{-1}\, R \, D\; \eps^{-1})}.
   \end{equation}
 
  This leads to
  \begin{equation} \label{eq: est_tv_scond}
     \norm{\nu H_\rho^n - \pi_\rho}_{\mbox{tv}} \leq \eps/2 + 
	D\, \exp\left({\frac{ -10^{-26}\,n}{8\,(d r^{-1}\, R)^2 \log^2(4\,d r^{-1}\, R\,D\; \eps^{-1})} }\right).
  \end{equation}

\item Now let us assume that $S_\eps \neq \emptyset$:\\
Let $\tilde{\eps}:=\nu(S_\eps)$, so that $0<\tilde{\eps}\leq \eps \leq 1/2$ and for $C\in\Borel(K)$ let
\[
  \mu_1(C)=\frac{\nu(C\cap S_\eps^c)}{\nu(S_\eps^c)} \quad \mbox{and} \quad \mu_2(C)=\frac{\nu(C\cap S_\eps)}{\nu(S_\eps)}.
\]
Then 
\[
  \nu = (1-\tilde{\eps})\mu_1 + \tilde{\eps} \mu_2
\]
and $\norm{\frac{d\mu_1}{d\pi_\rho}}_{\infty}\leq 2 D$. 
Furthermore for any $C\in\Borel(K)$ we find
\begin{eqnarray*}
 \abs{\nu H^n_\rho(C)-\pi_\rho(C)} \leq (1-\tilde{\eps}) \abs{\mu_1 H^n_\rho(C)-\pi_\rho(C)} + \tilde{\eps}.
\end{eqnarray*}
By using $(\ref{eq: est_tv_scond})$ we get
\[
\norm{\mu_1 H_\rho^n - \pi_\rho}_{\mbox{tv}} \leq  \eps/2 + 
	2\,D \exp\left({\frac{ -10^{-26}\,n}{8\,(d r^{-1}\, R)^2 \log^2(8\,d r^{-1}\, R\,D\, \eps^{-1})} }\right),
\]
and altogether 
\begin{equation} \label{eq: final_est_tv}
  \norm{\nu H_\rho^n - \pi_\rho}_{\mbox{tv}} \leq 3\,\eps/2 + 2\,D 
   \exp\left({\frac{ -10^{-26}\,n}{8\,(d r^{-1}\, R)^2 \log^2(8\,d r^{-1}\, R\,D\, \eps^{-1})} }\right).
\end{equation}
Choosing $n$ so that the right hand side of the previous equation is less than or equal to $2\eps$ completes the proof.
\end{enumerate}
\end{proof}
The next Corollary provides an explicit upper bound of the total variation distance.

\begin{corollary} \label{coro: est_har_tv}
Under the assumptions of Theorem~\ref{thm: est_n_tv} with
 \[
    \beta=\exp\left({\frac{-10^{-9}}{(d r^{-1}\, R)^{2/3}}}\right) \quad
\mbox{and} \quad C=12\,dr^{-1}\, R\,D
 \]
one obtains
    \[
    \norm{\nu H^{n_0}_\rho-\pi_\rho}_{\mbox{tv}} \leq C\; \beta^{\sqrt[3]{n_0}}, \quad n\in\N.
   \]
\end{corollary}
\begin{proof}
 Set $\eps=8\, d r^{-1}\, R\, D \exp\left(\frac{-10^{-9} \; n^{1/3}}{(dr^{-1}\, R)^{2/3}}\right)$ and use $(\ref{eq: final_est_tv})$ to complete the proof. 
\end{proof}
Note that the result of Theorem~\ref{thm: est_n_tv} is better than the result of Corollary~\ref{coro: est_har_tv}.
However, Corollary~\ref{coro: est_har_tv} provides an explicit estimate of the total variation distance. One can see that
there is an almost exponential decay, namely the total variation distance goes to zero at least as $\beta^{\sqrt[3]{n_0}}$ goes to zero for 
increasing $n_0$. 

In the previous results we assumed that $\rho \in \mathcal{U}_{r,R,\kappa}$. 
It is essentially used that (\ref{en: bound}) holds. 
Now let us assume that $\rho\in \mathcal{V}_{r,R,\kappa}$. 
The next statement is proven in \cite[Theorem~1.1]{LoVe06-1}.

\begin{theorem} \label{thm: gen_tv_est}
Let $\eps\in(0,1/2)$, $\rho\in \mathcal{V}_{r,R,\kappa}$.
Let $\nu$ be an initial distribution with the following property. 
There exists a set $S_\eps\subset K$ and a number $D\geq1$ such that
 \[
    \frac{d\nu}{d\pi_{\rho}}(x) \leq D, \quad x\in K\setminus S_\eps,
 \]
 where $\nu(S_\eps)\leq \eps$. Then for 
 \[
    n_0 \geq 4\cdot 10^{30} (d r^{-1}\, R)^2 \log^2(2\; D\, d r^{-1}\, R\,\eps^{-1} ) \log^3(D\, \eps^{-1})
 \]
 the total variation distance between $\;\nu H_{\rho}^{n_0}$ and $\pi_\rho$ is less than $2\eps$.
\end{theorem}

Note that Theorem~\ref{thm: est_n_tv} and Theorem~\ref{thm: gen_tv_est} 
can be applied if the initial distribution is bounded, 
i.e. we can set $D=\norm{\frac{d\nu}{d\pi_\rho}}_{\infty}$ and $S_\eps=\emptyset$.
Furthermore if $\nu \in \mathcal{M}_2$, i.e. $\norm{\frac{d\nu}{d\pi_\rho}}_{2}$ is bounded, then
we can also apply Theorem~\ref{thm: est_n_tv} and Theorem~\ref{thm: gen_tv_est} 
with $D= \norm{\frac{d\nu}{d\pi_\rho}}_{2}^2 \eps^{-1}$ and 
      \[
	S_\eps=\left\{ x\in K\mid \frac{d\nu}{d\pi_\rho}(x) >  \norm{\frac{d\nu}{d\pi_\rho}}_{2}^2 \eps^{-1} 	\right\}.
      \]

\section{Main results}  \label{sec: main_res}
Now we are able to state and to prove the main results. To avoid any pathologies we assume that $r^{-1} R d \geq 3$.

\begin{theorem}   \label{thm: Err}
 Let $\eps\in(0,1/2)$ and  
  \[
    \mathcal{F}_{r,R,{\kappa}}=
  \left\{ (f,\rho) 
  \mid \rho\in \mathcal{U}_{r,R, {\kappa}},\;
      \norm{f}_{\infty}\leq 1 \right\}.
  \]
 For $(f,\rho)\in \mathcal{F}_{r,R,{\kappa}}$ let $\nu$ be the uniform distribution on $G\subset \R^d$ from (\ref{en: init}). 
 Let $X^1_{n_0},\dots,X_{n_0}^n$ 
 be a sequence of the $n_0$th steps of $n$ independent hit-and-run Markov chains with stationary distribution $\pi_\rho$
 and initial distribution $\nu$.
 Recall that
\[
   M_{n,n_0}(f,\rho)   = \frac{1}{n}  \sum_{j=1}^n f(X_{n_0}^j).  
\]
 Then for  $n\geq \eps^{-2}$ and
 \[
    n_0\geq 10^{27} (d r^{-1}\, R)^2 \log^2(8\,d r^{-1}\, R\,\kappa\; \eps^{-2})\log(4\kappa \;\eps^{-2})
 \]
 we obtain  
 \[
    \sup_{(f,\rho)\in\mathcal{F}_{r,R,{\kappa}}} e(M_{n,n_0}(f,\rho)) \leq 3\eps.
 \]
 Hence 
  \[
    t_{\eps,{\kappa}}(n,n_0)
 =\mathcal{O} (d^{2}\,(r^{-1}\, R)^2\,\log^2(d r^{-1}\, R) \,\eps^{-2}\,[\log \eps^{-1}]^3
		    \, [\log\kappa]^3).
  \]

\end{theorem}
\begin{proof}
 For $C\in\Borel(K)$ we have
  \[
    \nu(C) = \int_C \frac{\mathbf{1}_G(y) \int_K \rho(x)\, \dint x}{\vol_d(G) \rho(y)} \,\pi_\rho(\dint y).
  \]
 It implies that $\frac{d\nu}{d\pi_\rho} (x) \leq \kappa$ for all $x\in K$. 
 Then the assertion follows by Theorem~\ref{thm: err_bound_m} and Theorem~\ref{thm: est_n_tv}.
\end{proof}

Now let us consider densities which belong to $\mathcal{V}_{r,R,\kappa}$.

\begin{theorem}  \label{thm: gen_Err}
Let $\eps\in(0,1/2)$ and  
  \[
    \mathcal{G}_{r,R,{\kappa}}=
  \left\{ (f,\rho) 
  \mid \rho\in \mathcal{V}_{r,R, {\kappa}},\;
      \norm{f}_{\infty}\leq 1 \right\}.
  \]
 Let $M_{n,n_0}$ be given as in Theorem~\ref{thm: Err}.
 Then for  $n\geq \eps^{-2}$ and
 \[
    n_0\geq 4\cdot 10^{30} (d r^{-1}\, R)^2 \log^2(2\,d r^{-1}\, R\,\kappa\; \eps^{-2})\log^3(\kappa \;\eps^{-2})
 \]
 we obtain  
 \[
    \sup_{(f,\rho)\in\mathcal{G}_{r,R,{\kappa}}} e(M_{n,n_0}(f,\rho)) \leq 3\eps.
 \]
 Hence 
  \[
    t_{\eps,{\kappa}}(n,n_0)
 =\mathcal{O} (d^{2}\,(r^{-1}\, R)^2\,\log^2(d r^{-1}\, R) \,\eps^{-2}\,[\log \eps^{-1}]^5
		    \, [\log\kappa]^5).
  \]
\end{theorem}
\begin{proof}
 The assertion follows by the same steps as the proof of Theorem~\ref{thm: Err}. 
 Note that we use Theorem~\ref{thm: gen_tv_est} instead of Theorem~\ref{thm: est_n_tv}.
\end{proof}

Note that in both theorems there is no hidden dependence on further parameters in the $\mathcal{O}$ notation. 
However,
the explicit constant might be very large, of the magnitude of $10^{30}$.
The theorems imply that the problem of integration (\ref{eq: sol}) is tractable with respect to $\kappa$
on the classes $\mathcal{F}_{r,R,\kappa}$ and $\mathcal{G}_{r,R,\kappa}$.\\

\emph{Example of a Gaussian function revisited}.
In the Gaussian example of Section~\ref{sec: densities} we obtained
\begin{eqnarray*}
 R/r & = & (2\,r^*(d)^{1/2})^{-1}\cdot\sqrt{{\tr(\Sigma)}/{\lambda_{\rm{min}}} } ,\\
 \kappa & = &\exp(\frac{1}{2}\;\lambda_{\rm{min}}^{-1})\; \Gamma(d/2+1)\;2^{d/2}\sqrt{\;\det(\Sigma)}.
\end{eqnarray*}
If we assume that $r^*(d)$ increases linearly in $d$ (Figure~\ref{fig: r_star_d}), 
that $\sqrt{{\tr(\Sigma)}/{\lambda_{\rm{min}}} } $ 
and $\log(\exp(\frac{1}{2}\;\lambda_{\rm{min}}^{-1})\sqrt{\;\det(\Sigma)})$
grows polynomially in the dimension, then $ t_{\eps,{\kappa}}(n,n_0) $ 
grows also polynomially in the dimension.
This implies that the integration problem with respect to the Gaussian function 
is polynomially tractable in the sense of Novak and Wo{\'z}niakowski \cite{NoWo08,NoWo10,NoWo12}.

\begin{acknowledgement}
The author gratefully acknowledges the comments of the referees and wants
to express his thanks to the local organizers
of the Tenth International Conference on Monte Carlo and Quasi-Monte Carlo
Methods in Scientific Computing for their hospitality.
The research was supported by the DFG Priority Program 1324 and the DFG Research Training Group 1523.
\end{acknowledgement}

\bibliographystyle{spmpsci}


\end{document}